\documentclass[11pt,twosided]{article}
\usepackage{amsmath}
\usepackage{amsthm}
\usepackage{amsfonts}
\usepackage{amssymb}
\usepackage{graphicx}

\theoremstyle{plain}
\newtheorem{thm}{Theorem}[section]
\theoremstyle{definition} \theoremstyle{definition}
\newtheorem{defn}[thm]{Definition}

\newtheorem{coro}[thm]{Corollary}
\newtheorem{prop}[thm]{Proposition}

\renewcommand\thefootnote{\fnsymbol{footnote}}

\numberwithin{equation}{section}
\begin{document}

\title{{\huge \textbf{On Geometry of Isophote Curves in Galilean space}}}
\author{ Z\"{u}hal K\"{u}\c{c}\"{u}karslan Y\"{u}zba\c{s}\i\ $^\dag$ and Dae Won Yoon $^\ddag$ }
\date{}
\maketitle
{\footnotesize \vskip 0.2 cm {%
\centerline  { $\dag$ Department of
Mathematics}}} {\footnotesize {\centerline  {F\i rat University }}}

{\footnotesize \centerline{ 23119 Elazig, Turkey} }

{\footnotesize 
\centerline { {E-mail address}:   {\tt
zuhal2387@yahoo.com.tr}} }

{\footnotesize {%
\centerline  {$\ddag$ Department of Mathematics Education and
RINS}}}

{\footnotesize \centerline  {Gyeongsang National University }}

{\footnotesize \centerline{ Jinju 52828, Republic of Korea} }

{\footnotesize \centerline { {E-mail address}:   {\tt dwyoon@gnu.ac.kr}} }

\begin{abstract}
In this paper, we introduce isophote curves on surfaces in Galilean 3-space.
Apart from the general concept of isophotes, we split our studies into two cases
to get the axis $d$ of isophote curves lying on a surface such that $d$ is
an isotropic or a non isotropic vector. We also give the method to compute
isophote curves of surfaces of revolution. Subsequently, we show the
relationship between isophote curves and slant(general) helices on surfaces
of revolution obtained by revolving a curve by Euclidean rotations. Finally,
we give an example to compute isophote curves on isotropic surfaces of
revolution.
\end{abstract}

\renewcommand\leftmark {\centerline{  \rm On Geometry of Isophote Curves }}
\renewcommand\rightmark {\centerline{ \rm On Geometry of Isophote Curves }}

\renewcommand{\thefootnote}{}
\footnote{{ {Corresponding author:}} Dae Won Yoon.}
 \footnote{%
}

\footnote{
\par
The second author was supported by Basic Science Research Program through
the National Research Foundation of Korea funded by the Ministry of
Education(NRF-2018R1D1A1B07046979).}

%%%%%%%%%%%%%%%%%%%%%%%%%%%%%%%%%%%%%%%%%%%%%%%%%%%%%%%%%

\section{Introduction}

\renewcommand{\theequation}{1.\arabic{equation}} \setcounter{equation}{0}

The isophote curve method is one of the most efficient methods that can be
used to analyze and visualize surfaces by lines of equal light intensity.
Isophote curve whose normal vectors make a constant angle with a fixed
vector(the axis) is one of the curves to characterize surfaces such as
parameter, geodesics and asymptotic curves or lines of curvature. Moreover,
this curve is used in computer graphics and it is also interesting to study
for geometry.

The isophote curve of a given surface is calculated with two steps: firstly
the normal vector field $n(s,t)$ of the surface is computed, and secondly
the surface point is traced as%
\begin{equation*}
\frac{\left\langle n(s,t),d\right\rangle }{\left\Vert n(s,t)\right\Vert }%
=\cos \beta ,\text{ }
\end{equation*}%
where $\beta $ is a constant angle($0\leq \beta \leq \frac{\pi }{2}$).

Isophote curve is called a silhouette curve when the angle $\beta $ is given
as a right angle such that \ 
\begin{equation*}
\frac{\left\langle n(s,t),d\right\rangle }{\left\Vert n(s,t)\right\Vert }%
=\cos \frac{\pi }{2}=0,
\end{equation*}%
where $d$ is the fixed vector.

From past to present, there have been a lot of researchers about isophote
curves and their characterizations in \cite{dog,dog2,kim,koe}.

In this paper, our aim is to investigate isophote curves on surfaces in
Galilean space and find its axis $d$ such that it is an isotropic and a non
isotropic vector by means of the Galilean Darboux frame. According to the
axis $d $, we split our studies into two cases to find the axis of isophote
curves lying on a surface in Galilean space. Moreover, we give the method to
compute isophote curves of surfaces of revolution obtained by revolving a
curve by Euclidean and isotropic rotations.

\section{Preliminaries}

\renewcommand{\theequation}{2.\arabic{equation}} \setcounter{equation}{0}

In accordance with the Erlangen Program, due to F. Klein, each geometry is
associated with a group of transformations, and hence there are as many
geometries as groups of transformations. Associated with group of
transformations that in physics guarantees the invariance of many mechanical
systems, the Galilei group, is the so-called Galilean geometry. That is,
Galilean geometry is one of the nine Cayley-Klein geometries with projective
signature $(0,0,+,+)$. The absolute of the Galilean geometry is an ordered
triple $\{ \omega, f, I \}$, where $\omega$ is the ideal (absolute) plane, $%
f $ the line in $\omega$ and $I$ the fixed elliptic involution of $f$.

We introduce homogeneous coordinates in ${G}_{3}$ in such a way that the
absolute plane $\omega $ is given by $x_{0}=0$, the absolute line $f$ by $%
x_{0}=x_{1}=0$ and the elliptic involution by $(0:0:x_{2}:x_{3})\rightarrow
(0:0:x_{3}:-x_{2})$.

The group of motions of $G_{3}$ is a six-parameter group given (in affine
coordinates) by 
\begin{equation*}
\begin{aligned} \bar x & = a + x,\\ \bar y & = b + c x + y \cos \varphi + z
\sin \varphi,\\ \bar z & = d + e x - y \sin\varphi + z \cos \varphi.
\end{aligned}
\end{equation*}

A plane is called Euclidean if it contains $f$, otherwise it is called
isotropic or i.e., planes $x=consant$ are Euclidean, and so is the plane $%
\omega $. Other planes are isotropic. In other words, an isotropic plane
does not involve any isotropic direction.

A Galilean scalar product of two vectors $x=(x_{1},y_{1},z_{1})$ and $%
y=(x_{2},y_{2},z_{2})$ in the Galilean 3-space $G_{3}$ is defined as 
\begin{equation*}
\langle x,y\rangle =%
\begin{cases}
x_{1}x_{2},\qquad \qquad & \text{if }~~x_{1}\neq 0\quad \text{or}\quad
x_{2}\neq 0, \\ 
y_{1}y_{2}+z_{1}z_{2},\quad & \text{if}\quad x_{1}=0\quad \text{and}\quad
x_{2}=0%
\end{cases}%
\end{equation*}%
and a Galilean norm of $x$ is given by 
\begin{equation*}
||x||=%
\begin{cases}
|x_{1}|,\qquad \qquad & \text{if }~~x_{1}\neq 0, \\ 
\sqrt{y_{1}^{2}+z_{1}^{2}},\quad & \text{if}\quad x_{1}=0.%
\end{cases}%
\end{equation*}

A Galilean cross product of $x$ and $y$ on $G_{3}$ is defined by 
\begin{equation*}
x\times y=%
\begin{vmatrix}
0 & e_{2} & e_{3} \\ 
x_{1} & y_{1} & z_{1} \\ 
x_{2} & y_{2} & z_{2}%
\end{vmatrix}%
,
\end{equation*}%
where $e_{2}=(0,1,0)$ and $e_{3}=(0,0,1)$, \cite{mol,ro}.

Let $\alpha $ be an admissible curve of the class ${C}^{\infty }$ in ${G}%
_{3},$ and parametrized by the invariant parameter $s$, defined by%
\begin{equation*}
\alpha (s)=\left( s,f(s),g(s)\right) .
\end{equation*}%
Then the Frenet frame fields of $\alpha (s)$ are given by 
\begin{equation*}
\begin{aligned} T(s) &= \alpha'(s),\\ N(s) &= \frac {1}{\kappa(s)}
\alpha''(s),\\ B(s) &= T(s) \times N(s), \end{aligned}
\end{equation*}%
where the curvature $\kappa (s)$ and the torsion $\tau (s)$ of $\alpha (s)$
are written as, respectively, 
\begin{eqnarray*}
\kappa \left( s\right) &=&\sqrt{f^{\prime \prime }\left( s\right)
^{2}+g^{\prime \prime }\left( s\right) ^{2}}, \\
\tau \left( s\right) &=&\frac{\det \left( \alpha ^{\prime }\left( s\right)
\alpha ^{\prime \prime }\left( s\right) \alpha ^{\prime \prime \prime
}\left( s\right) \right) }{\kappa ^{2}\left( s\right) }.
\end{eqnarray*}%
Here $T,$ $N$ and $B$ are said to be the tangent, principal normal and
binormal vectors of $\alpha (s)$.

\noindent On the other hand, the Frenet formula of the curve is given by
(cf. \cite{Pav})%
\begin{equation}
\begin{aligned} T^{\prime } &=&\kappa N, \\ N^{\prime } &=&\tau B, \\
B^{\prime } &=&-\tau N. \end{aligned}  \label{B3}
\end{equation}

Consider a $C^{r}$-regular surface $M$, $r\geq 1$, in $G_{3}$ parameterized
by 
\begin{equation*}
\mathbf{X}(u_{1},u_{2})=(x(u_{1},u_{2}),y(u_{1},u_{2}),z(u_{1},u_{2})).
\end{equation*}%
We denote by $x_{u_{i}}$, $y_{u_{i}}$ and $z_{u_{i}}$ the partial
derivatives of the functions $x$, $y$ and $z$ with respect to $u_{i}$ ($%
i=1,2 $), respectively.

On the other hand, the matrix of the first fundamental form $ds^{2}$ of a
surface $M$ in $G_{3}$ is given by 
\begin{equation*}
ds^{2}=%
\begin{pmatrix}
ds_{1}^{2} & 0 \\ 
0 & ds_{2}^{2}%
\end{pmatrix}%
,
\end{equation*}%
where $ds_{1}^{2}=(g_{1}du_{1}+g_{2}du_{2})^{2}$ and $%
ds_{2}^{2}=h_{11}du_{1}^{2}+2h_{12}du_{1}du_{2}+h_{22}du_{2}^{2}$. Here $%
g_{i}=x_{u_{i}}$ and $h_{ij}=\langle \tilde{\mathbf{X}}_{u_{i}},\tilde{%
\mathbf{X}}_{u_{j}}\rangle $ $(i,j=1,2)$ means the Euclidean scalar product
of the projections $\tilde{\mathbf{X}}_{u_{i}}$ of vectors $\mathbf{X}%
_{u_{i}}$ onto the $yz$-plane.

The unit normal vector field $n$ of a surface $M$ is defined by 
\begin{equation*}
n=\frac{1}{\omega }%
(0,x_{u_{2}}z_{u_{1}}-x_{u_{1}}z_{u_{2}},x_{u_{1}}y_{u_{2}}-x_{u_{2}}y_{u_{1}}),
\end{equation*}%
where the positive function $\omega $ is given by 
\begin{equation*}
\omega =\sqrt{%
(x_{u_{2}}z_{u_{1}}-x_{u_{1}}z_{u_{2}})^{2}+(x_{u_{1}}y_{u_{2}}-x_{u_{2}}y_{u_{1}})^{2}%
}.
\end{equation*}

%%%%%%%%%%%%%%%%%%%%%%%%%%%%%

Let $\{T,Q,n\}$ be a Galilean Darboux frame of $\alpha (s)$ with $T$ as the
tangent vector of a curve $\alpha (s)$ in $G_{3}$ and $n$ be the unit normal
to a surface and $Q=n\times T$. Then the Galilean Darboux frame is expressed
as 
\begin{equation}
\begin{aligned} T^{\prime } &=k_{g}Q+k_{n}n, \\ Q^{\prime } &=\tau _{g}n, \\
n ^{\prime } &=-\tau _{g}Q, \end{aligned}  \label{B4}
\end{equation}%
where $k_{g}$, $k_{n}$ and $\tau _{g}$ are the geodesic curvature, normal
curvature and geodesic torsion of $\alpha (s)$ on $M$, respectively. Also, %
\eqref{B4} implies 
\begin{equation}
\begin{aligned} \kappa ^{2}& =k_{g}^{2}+k_{n}^{2},\quad \tau =-\tau
_{g}+\frac{k_{g}^{\prime }k_{n}-k_{g}k_{n}^{\prime
}}{k_{g}^{2}+k_{n}^{2}},\\ k_{g} & =k\cos \phi \text{ and }k_{n}=-k\sin
\phi, \end{aligned}  \label{B5}
\end{equation}%
where $\phi $ is an angle between the surface normal vector $n$ and the
binormal vector $B$ of $\alpha ,$ (\cite{sa}). A curve $\alpha (s)$ is a
geodesic (an asymptotic curve or a line of curvature) if and only if $k_{g}$
( $k_{n}$ or $\tau _{g}$) vanishes, respectively.

On the other hand, the usual transformation between the Galilean Frenet
frames and the Darboux frames takes the form 
\begin{eqnarray}
Q &=&\cos \phi N+\sin \phi B,  \label{B6} \\
n &=&-\sin \phi N+\cos \phi B.  \notag
\end{eqnarray}

Artykbaev was introduced an angle between two vectors in Galilean space as
follows:

\begin{defn}
(\cite{Art}) Let $x=(1,x_{2},x_{3})$ and $y=(1,y_{2},y_{3})$ be two unit
non-isotropic vectors in $G{_{3}}$. Then an angle $\vartheta $ between $x$
and $y$ is defined by%
\begin{equation}
\vartheta =\sqrt{(y_{2}-x_{2})^{2}+(y_{3}-x_{3})^{2}}.  \label{B7}
\end{equation}
\end{defn}

\begin{defn}
(\cite{Art}) An angle between a unit non-isotropic vector $x=(1,x_{2},x_{3})$
and an isotropic vector $y=(0,y_{2},y_{3})$ in $G{_{3}}$ is defined by 
\begin{equation}
\varphi =\frac{x_{2}y_{2}+x_{3}y_{3}}{\sqrt{y_{2}{}^{2}+y_{3}^{2}}}.
\label{B8}
\end{equation}
\end{defn}

\begin{defn}
(\cite{Art}) An angle $\theta $ between two isotropic vectors $%
x=(0,x_{2},x_{3})$ and $y=(0,y_{2},y_{3})$ parallel to the Euclidean plane
in $G{_{3}}$ is equal to the Euclidean angle between them. That is,%
\begin{equation}
\cos \theta =\frac{x_{2}y_{2}+x_{3}y_{3}}{\sqrt{x_{2}{}^{2}+x_{3}^{2}}\sqrt{%
y_{2}{}^{2}+y_{3}^{2}}}.  \label{B9}
\end{equation}
\end{defn}

\section{ The axis of an isophote curve in Galilean Space}

\renewcommand{\theequation}{3.\arabic{equation}} \setcounter{equation}{0}

The starting point of this section is to get the fixed vector $d$ of an
isophote curve via its Galilean Darboux frame.

Let $M$ be an admissible regular surface and $\alpha :I\subset R\rightarrow
M $ be an unit speed curve parametrized by $\alpha(s)=(s, \alpha_2(s),
\alpha_3(s))$ as an isophote curve for some $s \in I$.

In order to prove the results, we split it into two cases according to the
fixed vector $d$.

\vskip 0.2 cm

\textbf{Case 1.} $d$ is an unit isotropic vector.

Since $n$ is the unit isotropic normal vector of a surface $M$, we have 
\begin{equation}
\langle n,d\rangle =\cos \theta =\text{constant}.  \label{C1}
\end{equation}%
If we differentiate $\langle T,d\rangle =0$ with respect to $s$, using the
Galilean Darboux frame \eqref{B4}, then we obtain 
\begin{equation}
k_{g}\left\langle Q,d\right\rangle +k_{n}\langle n,d\rangle =0,  \label{C2}
\end{equation}%
which implies 
\begin{equation}
\left\langle Q,d\right\rangle =-\frac{k_{n}}{k_{g}}\cos \theta .  \label{C3}
\end{equation}%
Taking account of the derivative of \eqref{C1} we get 
\begin{equation}
\tau _{g}\left\langle Q,d\right\rangle =0,  \label{C4}
\end{equation}%
where if \ $\left\langle Q,d\right\rangle =0$, $k_{n}=0$ which means that $%
\alpha $ should be an asymptotic curve or $\tau _{g}=0$ which means that $%
\alpha $ should be a line of curvature. Then, for $k_{n}=0,$ $d$ can be
written as 
\begin{equation}
d=\cos \theta n,  \label{C5}
\end{equation}%
since $d$ is a constant vector, $\tau _{g}$ should be equal zero. Also this
is the trivial result.

For $\tau _{g}=0,d$ can be written as%
\begin{equation}  \label{C6}
d=-\frac{k_{n}}{k_{g}}\cos \theta Q+\cos \theta n.
\end{equation}
Since $\left\Vert d\right\Vert =1$, we get%
\begin{equation}  \label{C7}
\frac{k_{n}}{k_{g}}=\pm \tan \theta.
\end{equation}%
In this situation, we conclude that $\phi =\pm \theta $ or $\phi =\pi \pm
\theta .$

From \eqref{B5} and \eqref{B6} in terms of the Galilean Frenet frame, we get%
\begin{equation}
d=(-\frac{k_{n}}{k}\cos \theta -\frac{k_{g}}{k}\sin \theta )N+(-\frac{k_{n}}{%
k}\sin \theta +\frac{k_{g}}{k}\cos \theta )B.  \label{C8}
\end{equation}%
If we differentiate \eqref{C6} using \eqref{C7} and $\tau _{g}=0$, we get $%
d^{\prime }=0,$ that is, $d$ is a constant isotropic vector. From now on, we
suppose if $\alpha $ is a unit-speed isophote curve, then $\alpha $ is also
a line of curvature.

\begin{thm}
Let $\alpha $ be a unit-speed isophote curve on a surface $M$ in $G_3$ with
a fixed unit isotropic vector $d$ as the axis of the isophote curve. In that
case, we have the following:

$i)$ If $\alpha $ is a geodesic curve, then $\alpha $ is a straight line.

$ii)$ If $\alpha $ is an asymptotic curve on $M$, then it is a plane curve,
and the fixed vector $d$ is spanned by $B.$
\end{thm}

\begin{proof}\ $i)$ If $\alpha $ is a geodesic curve, then we have $k_{g}=0$ and
so from \eqref{C2} it follows that $k_{n}=0,$ also $\tau _{g}=0.$ By
substituting $k_{g}$ and $k_{n}$ into \eqref{B5}, we get $\kappa =0$, that is, $\alpha$ is a straight line.

$ii)$ If $\alpha $ is an asymptotic curve, we have $k_{n}=0.$ From %
\eqref{B5} and \eqref{C8}, we obtain that%
\begin{equation*}
d=\frac{k_{g}}{k}\cos \theta B.
\end{equation*}%
Also, by substituting $\tau _{g}=0$ and $k_{n}=0$ into \eqref{B6}, we get $%
\tau =0.$ It means that $\alpha $ is a plane curve.
\end{proof}

\begin{thm}
Let $\alpha $ be a unit-speed isophote curve on a surface $M$ in $G_3$ with
a fixed unit isotropic vector $d$ as the axis of the isophote curve. The
axis $d$ is perpendicular to the principal normal line of $\alpha $ if and
only if \ either $\alpha $ is a straight line, or an asymptotic curve on $M$
with taking $\dfrac{k_{n}}{k_{g}}=\tan \theta $ or $\alpha $ is a curve with 
$\frac{k_{n}}{k_{g}}=-\tan \theta .$
\end{thm}

\begin{proof} If $\alpha $ is a unit-speed isophote curve with $\frac{k_{n}}{k_{g}%
}=\tan \theta $, then from \eqref{C8}, we get
\begin{equation*}
\left\langle N,d\right\rangle =-2\frac{k_{g}}{k}\sin \theta =0,
\end{equation*}%
from this equation, we have $k_{g}=0$ or $\sin \theta =0$.

If  $k_{g}=0$  then, from Theorem 3.1,  $\alpha $ is a straight line.

If  $\sin \theta =0$, then  $k_{n}=0,$ that is,  $\alpha $ is an asymptotic curve.

 If we take $\frac{k_{n}}{k_{g}}=-\tan
\theta $, then we can easily get $\left\langle N,d\right\rangle =0.$
\end{proof}

\begin{thm}
Let $\alpha $ be a unit-speed isophote curve on a surface $M$ in $G_3$ with
a fixed unit isotropic vector $d$ as the axis of the isophote curve. The
axis $d$ is perpendicular to the principal binormal line of $\alpha $ such
that $\dfrac{k_{n}}{k_{g}}=\tan \theta$ if and only if \ $\theta $ equals $%
\dfrac{\pi }{4}$.
\end{thm}

\begin{proof}\ If $\alpha$ is a unit-speed isophote curve with $\frac{k_{n}}{k_{g}%
}=\tan \theta $, then from \eqref{C8}, we get
\begin{equation*}
\left\langle B,d\right\rangle =\frac{k_{g}}{k}(-\sin ^{2}\theta +\cos
^{2}\theta )=0.
\end{equation*}%
Since $\alpha $ is a non-geodesic curve,  $-\sin ^{2}\theta
+\cos ^{2}\theta =0.$ So, $\tan \theta =1$. We know that $0\leq\theta \leq
\dfrac{\pi }{2},$ then we get $\theta =\dfrac{\pi }{4}.$
\end{proof}

\begin{thm}
If $\alpha $ is a silhouette curve on $M,$ and $d$ is a unit isotropic
vector such that it is parallel to $Q$, then the curve $\alpha $ is a plane
curve.
\end{thm}

\begin{proof}
If a fixed vector $d$ is  a unit isotropic
vector and  is parallel to  $Q$, then we have
$$
d=\pm Q, \quad \left\langle T,d\right\rangle =0.
$$
By differentiating above equations with respect to $s$, we obtain
\begin{equation*}
\tau_g n=0, \quad k_g \langle Q, d\rangle + k_n \langle n, d\rangle =0.
\end{equation*}  
Since $\alpha$ is a silhouette curve with $\langle n, d\rangle =0$, we get 
$$
\tau_g=0, \quad k_g=0,
$$
from this, we have $\tau =0.$ It means that $\alpha $ is a plane curve.
\end{proof}

\textbf{Case 2.} Now, our aim is to find a fixed unit non-isotropic vector $%
d $ as the axis of an isophote curve.

Since $n$ is the unit isotropic normal vector of a surface $M$, we have 
\begin{equation}
\langle n,d\rangle =\varphi =\text{constant}.
\end{equation}

Let $\alpha $ be a unit speed admissible isophote curve. If we differentiate 
\begin{equation}
\left\langle T,d\right\rangle =1  \label{C9}
\end{equation}%
with respect to $s,$ using\ the Galilean Darboux frame \eqref{B4} then we
have 
\begin{equation}
k_{g}\left\langle Q,d\right\rangle +k_{n}\left\langle n,d\right\rangle =0.
\label{C10}
\end{equation}%
It follows from \eqref{B8} that we find 
\begin{equation}
\left\langle Q,d\right\rangle =-\frac{k_{n}}{k_{g}}\varphi .  \label{C11}
\end{equation}%
Taking account of the derivative of $\langle n,d\rangle =\varphi $ and using
the Galilean Darboux frame \eqref{B4} 
\begin{equation}
\tau _{g}\left\langle Q,d\right\rangle =0,  \label{C12}
\end{equation}%
where if \ $\left\langle Q,d\right\rangle =0$, then from \eqref{C11} we get $%
k_{n}=0$ which means that \ $\alpha $ should be an asymptotic curve. Then,
for $k_{n}=0,$ $d$ can be written as 
\begin{equation}
d=T+\varphi n.  \label{C13}
\end{equation}%
Since $d$ is a constant vector, $k_{g}=$ $\varphi \tau _{g}$. Thus, we have
the following result:

\begin{coro}
\textit{Let $\alpha $ be a unit-speed isophote curve on a surface $M$ in $%
G_3 $ with a fixed unit non-isotropic vector $d$ as the axis of the isophote
curve. If $\alpha$ is a geodesic curve or a line of curvature, then $\alpha$
is a straight line.}
\end{coro}

If $\tau _{g}=0$, that is, $\alpha $ is a line of curvature, then $d$ can be
written as%
\begin{equation}  \label{C14}
d=T-\frac{k_{n}}{k_{g}}\varphi Q+\varphi n.
\end{equation}%
Since $d$ is a constant vector, $k_{g}=$ $k _{n}=0$, which implies $\kappa=0$%
, that is, $\alpha $ is a straight line.

%\begin{thm}
% Let $\alpha $ be a unit-speed isophote curve on the surface $M$ .
%In that case, we have the following:

%$i)$ The axis isotropic $d$ is perpendicular to the principal normal line of
%$\alpha $ if and only if \ either $\alpha $ is an asymptotic curve$.$

%$ii)$ The axis isotropic $d$ is perpendicular to the principal normal line
%of $\alpha $ if and only if \ either $\alpha $ is an asymptotic curve and $%
%\theta =\frac{\pi }{2}+2k\pi .$
%\end{thm}

\begin{thm}
Let $\alpha $ be a silhouette curve on $M$ and $d$ be a unit non-isotropic
vector.

$i)$ If $d$ lies in the plane spanned by $T$ and $Q$, then $\alpha $ is a
plane curve. % Then $\alpha $ is a geodesic and a normal
%curvature if and only if it is a plane curve.

$ii)$ If the axis $d$ is spanned by $T$, then $\alpha $ is a geodesic curve.
\end{thm}

\begin{proof}
$i)$ Since  $\alpha $ is a silhouette curve and $d$ is a unit non-isotropic
vector, we get%
\begin{equation} \label{C15}
\left\langle T,d\right\rangle =\pm 1.
\end{equation}%
If we differentiate \eqref{C15} with respect to $s$,  then we get
\begin{equation*}
k_{g}\left\langle Q,d\right\rangle =0.
\end{equation*}%
Since  $d$ is lied in the plane spanned by $T$ and $Q,$ we get $k_{g}=0$. Also,
if we differentiate $\left\langle n,d\right\rangle =0$ with respect to $s$,
we get%
\begin{equation*}
\tau _{g}\left\langle Q,d\right\rangle =0,
\end{equation*}%
it follows that  $\tau _{g}=0.$

Also, by substituting $\tau _{g}=0$ and $k_{g}=0$ into \eqref{B5}, we get $%
\tau =0.$  Thus,  $\alpha $ is a plane curve.

%In this case, the fixed vector $d$ is expressed as
%\begin{equation*}
%d= T+aQ,
%\end{equation*}%
%where $a$ is a smooth function.

%\noindent In fact, $d$ is a constant vector. By differentiating above equation
%with respect to $s$, we get%
%\begin{equation*}
%d^{\prime }= k_{n}n+a^{\prime }Q,
%\end{equation*}%
%from this $k_{n}=0$ and $a$ is  constant. %Hence $\alpha $ is a plane curve.

$ii)$ If  $d$ is spanned by $T$,  then we get
\begin{equation*}
d= T.
\end{equation*}%
If we differentiate above equation,  then  $d^{\prime } = k_{g}Q$, it follows that   $k_{g} =0$, that is, the curve is a geodesic curve.
\end{proof}

\section{Applications for Isophote Curves}

\renewcommand{\theequation}{4.\arabic{equation}} \setcounter{equation}{0}

We investigate an isophote curve among surfaces in Galilean space. Now we
give some examples for this subject. To see this, notice that in $G_{3}$
surfaces of revolution are obtained by revolving a curve by Euclidean or
isotropic rotations as follows, respectively,%
\begin{eqnarray}
\overline{x} &=&x,  \label{D1} \\
\overline{y} &=&y\cos t+z\sin t,  \notag \\
\overline{z} &=&-y\sin t+z\cos t,  \notag
\end{eqnarray}%
where $t$ is the Euclidean angle and 
\begin{eqnarray}
\overline{x} &=&x+ct,  \label{D2} \\
\overline{y} &=&y+xt+c\frac{t^{2}}{2},  \notag \\
\overline{z} &=&z,  \notag
\end{eqnarray}%
where $t\in \mathbb{R}$ and $c=\text{constant}>0$.

The trajectory of a single point under a Euclidean rotation is a Euclidean
circle 
\begin{equation*}
x=\text{constant}, \quad y^2 + z^2 = r^2, \quad r \in \mathbb{R. }
\end{equation*}
The invariant $r$ is the radius of the circle. Euclidean circles intersect
the absolute line $f$ in the fixed points of the elliptic involution $(F_1,
F_2)$.

The trajectory of a point under isotropic rotation is an isotropic circle
whose normal form is 
\begin{equation*}
z=\text{constant},\quad y=\frac{x^{2}}{2c}.
\end{equation*}%
The invariant $c$ is the radius of the circle. The fixed line of the
isotropic rotation is the absolute line $f\ $\cite{Milin}. For some more
studies, see \cite{dede, kazan}.

If a curve $\alpha (s)=(f(s),0,g(s)),$ $\left( g(s)>0\right) $ is rotated by
Euclidean rotations, then a surface of revolution is parametrized by%
\begin{equation}
S(s,t)=(f(s),g(s)\sin t,g(s)\cos t).  \label{d3}
\end{equation}%
If a curve $\alpha (s)$ is parametrized by the arc-length, then we take $%
f(s)=s.$ Then, the unit isotropic normal vector field $n(s,t)$ of $S$ is
defined by 
\begin{equation}
n(s,t)=\frac{S_{s}\times S_{t}}{\left\Vert S_{s}\times S_{t}\right\Vert },
\label{D3}
\end{equation}%
where $S_{s}$ and $S_{t}$ are the partial differentiations with respect to $s
$ and $t,$ respectively. Then, the isotropic normal vector is given by%
\begin{equation*}
n(s,t)=(0,\sin t,\cos t),
\end{equation*}%
it becomes in terms of the Frenet frame as follows: 
\begin{equation}
n(s,t)=-\sin tB+\cos tN.  \label{D4}
\end{equation}%
%
%
%
%
%
%
%
%
%
%
%\end{exam}

\begin{prop}
\textit{Let a curve $\alpha (s)$ be a general helix with the isotropic axis $%
d$. Then, for $t_{0}=(\frac{2k+1}{2})\pi $ $(k\in Z),$ the curve $\alpha (s)$%
} on\textit{\ surfaces of revolution given by \eqref{d3} of revolution is an
isophote curve with the axis $d$.}
\end{prop}

\begin{proof}
 Substituting $t_{0}$ into \eqref{D4}, we get
\begin{equation*}
n(s,t_{0})=\mp B.
\end{equation*}%
If $\alpha(s)$ is a general helix with the axis $d$, then $\left\langle
B,d\right\rangle =$constant. Therefore, we get%
\begin{equation*}
\left\langle n(s,t_{0}),d\right\rangle =\mp \left\langle B,d\right\rangle =%
\text{constant.}
\end{equation*}
Thus $\alpha(s)$ is an  isophote curve with the axis $d$ on the surfaces of
revolution.
\end{proof}

\begin{prop}
\textit{Let a curve $\alpha (s)$ be a slant helix with the isotropic axis $d$%
. Then, for $t_{0}=k\pi $ $(k\in Z),$ the curve $\alpha (s)$ on surfaces of
revolution given by \eqref{d3} is an isophote curve with the axis $d$.}
\end{prop}

\begin{proof}
 Substituting $t_{0}$ into \eqref{D4}, we get
\begin{equation*}
n(s,t_{0})=\mp N.
\end{equation*}%
If $\alpha(s)$ is a slant helix with the axis $d$, then $\left\langle
N,d\right\rangle =$constant. Therefore, we get%
\begin{equation*}
\left\langle n(s,t_{0}),d\right\rangle =\mp \left\langle N,d\right\rangle =%
\text{constant.}
\end{equation*}
Thus $\alpha(s)$ is an isophote curve with the axis $d$ on the surfaces of
revolution.
\end{proof}

If a curve $\alpha (s)=(f(s),0,g(s)),$ $\left( g(s)>0\right) $ is rotated by
isotropic rotations, then a surface of revolution is parametrized by%
\begin{equation}
S(s,t)=(f(s)+ct,st+c\frac{t^{2}}{2},g(s)).  \label{d4}
\end{equation}%
If a curve $\alpha (s)$ is parametrized by the arc-length, then we take $%
f(s)=s.$ Then, the isotropic surface normal is given by%
\begin{equation*}
n=\frac{1}{\sqrt{\left( g^{\prime }(s)c\right) ^{2}+s^{2}}}(0,g^{\prime
}(s)c,s),
\end{equation*}%
it becomes in terms of the Frenet frame as follows: 
\begin{equation}
n=\frac{1}{\sqrt{\left( g^{\prime }(s)c\right) ^{2}+s^{2}}}\left( -g^{\prime
}(s)cB+sN\right) .  \label{D5}
\end{equation}

\begin{prop}
\textit{Let an isotropic axis $d$ is given by (0,}$\mathit{d}_{y},d_{z}$%
\textit{). }

$i)$ \textit{If $\mathit{d}_{y}=0$ and $g(s)$ is a second order function, then}%
\textit{\ the curve }$\alpha (s)$\textit{ on surfaces of revolution given
by \eqref{d4} is an isophote curve.}

$ii)$ \textit{If $\mathit{d}_{z}=0$ and $g(s)$ is a second order function, then the
curve} $\alpha (s)$\textit{ on surfaces of revolution given by %
\eqref{d4} is an isophote curve.}
\end{prop}

\begin{proof}

$i)$ If $\mathit{d}_{y}=0,$\textit{\ then we get }$d=\lambda_1N, ( \lambda_1\in R_0) $.

Using this above condition on \eqref{D5}, we get 
\begin{equation*}
\left\langle n,d\right\rangle =\frac{\lambda_1s}{\sqrt{\left( g^{\prime }(s)c\right)
^{2}+s^{2}}}.
\end{equation*}%
From the above equation, we can get $g(s)=\frac{s^{2}}{2c}+A,$ $A\in R.$
Thus we obtain $\left\langle n,d\right\rangle =\frac{\lambda_1}{\sqrt2}.$

$ii)$ If $\mathit{d}_{z}=0,$\textit{\ then we get }$d=-\lambda_2B, ( \lambda_2\in R_0).$
Using this above condition on \eqref{D5}, we get 
\begin{equation*}
\left\langle n,d\right\rangle =\frac{ \lambda_2g^{\prime }(s)c}{\sqrt{\left( g^{\prime
}(s)c\right) ^{2}+s^{2}}}.
\end{equation*}%
From the above equation, we can get $g(s)=\frac{s^{2}}{2c}+A,$ $A\in R.$
Thus we obtain $\left\langle n,d\right\rangle =\frac{\lambda_2}{\sqrt2}.$
\end{proof}

\begin{figure}[tbp]
\centering
\includegraphics[width=4cm,height=6cm]{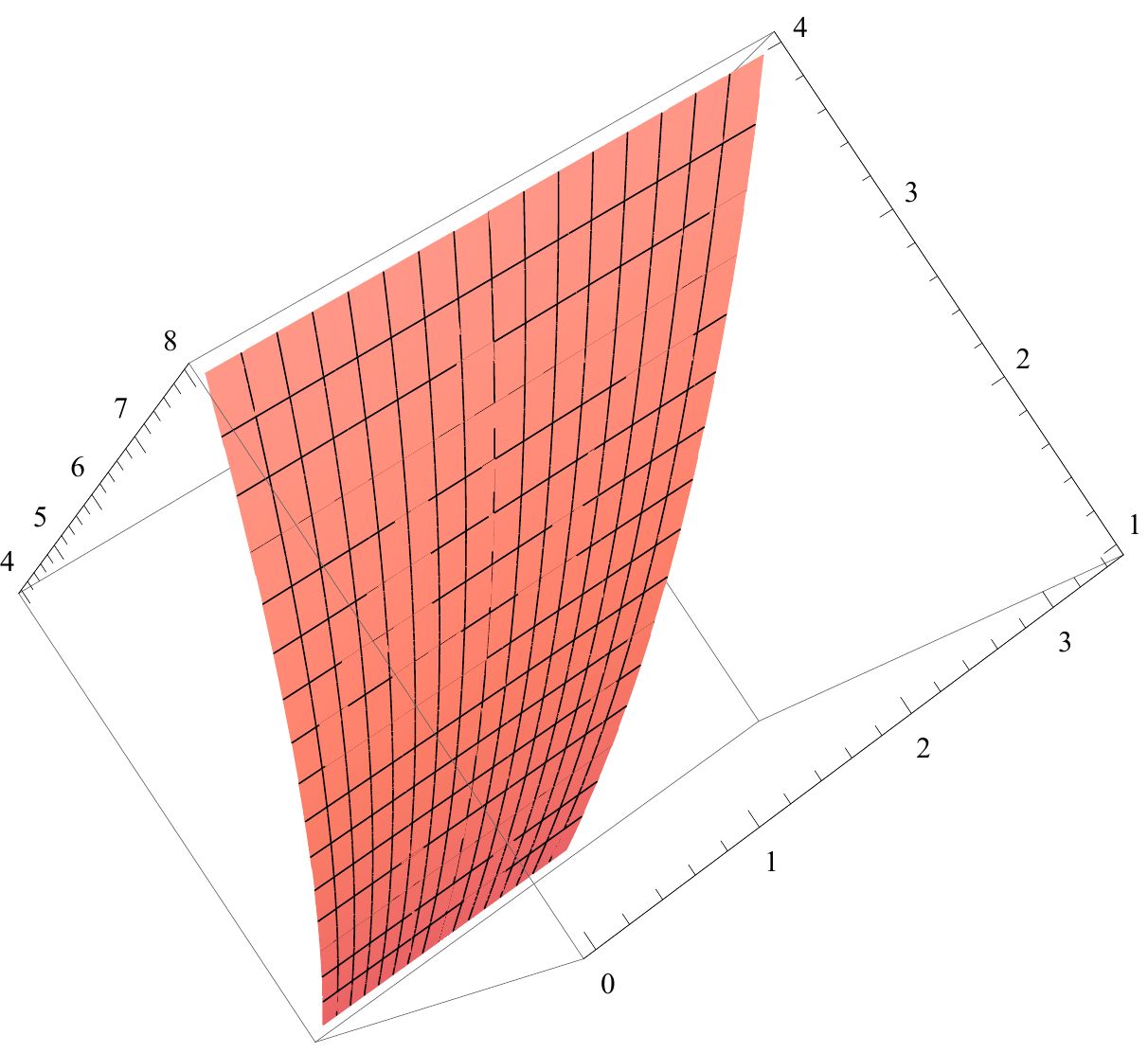} \label{fig:gull}
\caption{ Isotropic surface of revolution for $c=1$ and $A=0$. }
\label{fig:animals}
\end{figure}
Therefore, the rotating curve is an isotropic circle on surfaces of
revolution. We also show the surfaces \eqref{d4} for $g(s)=\frac{s^{2}}{2c}+A
$ in Figure 1.

\begin{coro}
The generating curve $\alpha (s)=(f(s),0,g(s))$ on surfaces of revolution
given by \textit{\eqref{d4}} becomes both a general helix and a slant helix
with the axis $d$.
\end{coro}


\begin{thebibliography}{99}
\bibitem{Art} A. Artykbaev, Total angle about the vertex of a cone in
Galilean space, Math. Notes, \textbf{43} (1988), 379--382

\bibitem{dede} M. Dede, C. Ekici and W. Goemans, Surfaces of revolution with
vanishing curvature in Galilean 3-space, J. Math. Phys. Anal. Geo., \textbf{%
14}(2) (2018), 141-152.

\bibitem{dog} F. Do\u{g}an and Y. Yayl\i, On isophote curves and their
characterizations, Turkish J. Math., \textbf{39}(5) (2015), 650-664.

\bibitem{dog2} F. Do\u{g}an and Y. Yayl\i , Isophote curves on spacelike
surfaces in Lorentz-Minkowski space $E_{1}^{3}$, arXiv preprint
arXiv:1203.4388.

\bibitem{kazan} A. Kazan and H. B. Karadag, Weighted Minimal and Weighted
Flat Surfaces of Revolution in Galilean 3-Space with Density,  Int. J. Anal.
Appl., \textbf{16}(3) (2018), 414-426.

\bibitem{kim} K. J. Kim and I. K. Lee, Computing isophotes of surface of
revolution and canal surface, Comput-Aided Des., \textbf{35} (2003), 215-223.

\bibitem{koe} J. J. Koenderink, A. J. van Doorn, Photometric invariants
related to solid shape, J. Modern Opt., \textbf{\ 27}(7) (1980), 981-996.

\bibitem{mol} E. Molnar, The projective interpretation of the eight
3-dimensional Homogeneous geometries, Beitr. Algebra Geom., \textbf{38}
(1997), 261--288.

\bibitem{Pav} B. J. Pavkovic and I. Kamenarovic, The equiform differential
geometry of curves in the Galilean space G3, Glas. Mat., \textbf{22}(42)
(1987), 449-457.

\bibitem{ro} O. R\"{o}schel, Die Geometrie des Galileischen raumes,
Habilitationsschrift, Leoben, 1984.

\bibitem{Milin} Z. M. Sipus, Ruled Weingarten surfaces in Galilean space,
Period. Math. Hungar, \textbf{56}(2) (2008), 213--225.

\bibitem{sa} T. \c{S}ahin, Intrinsic equations for a generalized relaxed
elastic line on an oriented surface in the Galilean space, Acta Math. Sci., 
\textbf{33}(3) (2013), 701-711.

\bibitem{Yaglom} I. M. Yaglom, A simple non-Euclidean geometry and its
physical basis, Springer-Verlag: New York Inc, 1979.
\end{thebibliography}
\end{document}